\documentclass[11pt, reqno]{amsart}
\usepackage{amssymb}
\usepackage{amsmath}
\usepackage{mathrsfs}
\usepackage{amssymb, amsmath, amsfonts}

\usepackage{mathabx,url}

\usepackage{color}

\usepackage[hidelinks]{hyperref}

\makeatletter
\def\tank#1{\protected@xdef\@thanks{\@thanks
 \protect\footnotetext[0]{#1}}}
\def\bigfoot{

 \@footnotetext}
\makeatother

\newcommand{\ea}{\end{array}}

\allowdisplaybreaks
\numberwithin{equation}{section}
\newtheorem{theorem}{Theorem}[section]
\newtheorem{lemma}{Lemma}[section]
\newtheorem{proposition}{Proposition}[section]

\newtheorem{corollary}{Corollary}[section]

\newtheorem{cor}{Corollary}[section]

\newtheorem{example}{Example}[section]

\def\beq{\begin{equation}}
\def\nneq{\end{equation}}

\def\bthm{\begin{theorem}}
\def\nthm{\end{theorem}}

\def\blem{\begin{lemma}}
\def\nlem{\end{lemma}}
\def\bprf{\begin{proof}}
\def\nprf{\end{proof}}
\def\bprop{\begin{prop}}
\def\nprop{\end{prop}}
\def\brmk{\begin{rem}}
\def\nrmk{\end{rem}}

\def\bexa{\begin{exa}}
\def\nexa{\end{exa}}
\def\bcor{\begin{cor}}
\def\ncor{\end{cor}}

\title[Strassen's local LIL for GFBM]{Strassen's  local law of the iterated logarithm for the generalized fractional Brownian motion}

\author[R. Wang]{Ran Wang \textsuperscript{*}}
 \thanks{* Corresponding author.   School of Mathematics and Statistics,  Wuhan University,  Wuhan, 430072,
 China.  E-mail: rwang@whu.edu.cn}

\author[Y. Xiao]{Yimin Xiao \textsuperscript{$\#$}} 
\thanks{$\#$ Department of Statistics and Probability, Michigan State University, East Lansing,
 MI 48824, USA. E-mail: xiaoy@msu.edu}

 \date{}
\begin{document}
\maketitle
 \vskip-0.3cm
 \noindent {\bf Abstract:}
 Let $X:=\{X(t)\}_{t\ge0}$ be a generalized fractional Brownian motion given by
$$
 \{X(t)\}_{t\ge0}\overset{d}{=}\left\{  \int_{\mathbb R}  \left((t-u)_+^{\alpha}-(-u)_+^{\alpha} \right) |u|^{-\gamma/2} B(du)  \right\}_{t\ge0},
$$
with parameters $\gamma \in (0, 1)$ and $\alpha\in \left(-1/2+ 
\gamma/2, \,  1/2+\gamma/2\right)$. This 
process was introduced by Pang and Taqqu (2019) as the scaling
limit of a class of power-law shot noise processes. The
parameters $\alpha$ and $\gamma$ govern  the probabilistic and 
statistical properties of  $X$. In particular, the parameter 
$\gamma$  breaks the stationarity of increments of $X$.

In this paper, we establish Strassen's local law of the 
iterated logarithm for $X$ at a given point $t_0 \in (0, 
\infty)$. This result describes explicitly the roles played by
the parameters $\alpha, \gamma$, and the location $t_0$.  Our 
theorem  differs from the  earlier Strassen's  
{global law of the iterated logarithm} for $X$  proved by 
Ichiba, Pang and Taqqu (2022).
 \vskip0.3cm


 \vskip0.5cm



 \section{Introduction}


The generalized fractional Brownian motion (GFBM, in  short) $X:=\{X(t)\}_{t\ge0}$ is a centered Gaussian  self-similar
process introduced by Pang and Taqqu \cite{PT2019} as the scaling limit of a class of shot noise processes with power-law
non-stationary conditional variance functions. It has the following stochastic integral representation:
 \begin{align}\label{eq X}
 \left\{X(t)\right\}_{t\ge0}\overset{d}{=}&\left\{  \int_{\mathbb R}  \left((t-u)_+^{\alpha}-(-u)_+^{\alpha} \right) |u|^{-\gamma/2}
  B(du)  \right\}_{t\ge0},
 \end{align}
 where the parameters $\gamma$ and $\alpha$ satisfy
 \begin{align}\label{eq constant}
 \gamma\in  {(}0,1),  \ \  \alpha\in \left(-\frac12+\frac{\gamma}{2}, \  \frac12+\frac{\gamma}{2} \right),
 \end{align}
and   $B(du)$ is a Gaussian random measure on $\mathbb R$ with   Lebesgue  control measure $du$.
It follows that the Gaussian process $X$ is self-similar with index   given by
 \begin{align}\label{Eq:H}
 H=\alpha-\frac{\gamma}{2}+\frac12\in(0,1).
 \end{align}
When $\gamma=0$, $X$ becomes an ordinary fractional Brownian motion  (FBM, for short) $B^H$ which
can be represented as:
\begin{align*}
\left\{B^H(t)\right\}_{t\ge 0}\overset{d}{=}\left\{ \int_{\mathbb R}
\left((t-u)_+^{H-\frac12}-(-u)_+^{H-\frac12} \right) B(du)  \right\}_{t\ge0}.
\end{align*}
  When $\gamma \ne 0$, however, $X$ does not have the property of stationary increments.

As shown by Pang and Taqqu \cite{PT2019},  GFBM $X$ is a natural generalization of the ordinary FBM.
It preserves the self-similarity property while the factor $|u|^{-\gamma/2}$ introduces non-stationarity of increments.
Thus the parameter $\gamma$ not only makes the properties of the GFBM different from those of FBM but,
more importantly, also provides more flexibility in capturing the non-stationarity in stochastic systems.
 For  semimartingale properties of  the  GFBM and its mixtures in terms of $(\alpha,\gamma)$, as well as examples of applications in finance,  see Ichiba, Pang and Taqqu \cite{IPT2020b}.

The parameters $\alpha$ and $\gamma$ govern the probabilistic 
and statistical properties of the GFBM $X$, making it essential to understand their quantitative roles.  Ichiba, Pang  and Taqqu \cite{IPT2021} raised and investigated the question:  ``How does the  parameter $\gamma$ affect the sample path properties of GFBM?"    
 {They proved that the regularity of the GFBM $X$ depends critically on $\alpha$ and $H$. Specifically, for any $\alpha \in (-1/2 + \gamma/2, 1/2]$, the sample paths are H\"older continuous of order $H - \varepsilon$ on $[0, T]$ for any $T>0$ and $\varepsilon>0$; whereas for any $\alpha \in (1/2, 1/2 + \gamma/2)$, the process becomes continuously differentiable (see \cite[Theorems 3.1 and  4.1]{IPT2021}).} Moreover, both the local law of the iterated logarithm (LIL, for short) at the origin and the functional LIL (Strassen's global LIL) at infinity are governed solely by the self-similarity index $H$ (see \cite[Theorems 5.1, 6.1]{IPT2021} and \cite[Proposition 7.1]{WX2021a}). 

In \cite{WX2021a, WX2021b}, we have  
studied some precise sample path properties of the GFBM $X$, 
including the exact uniform modulus of continuity, small ball 
probabilities, Chung's law of the iterated logarithm, and
the tangent processes.  These results demonstrate that the 
local regularity properties of the GFBM $X$ away from the 
origin are determined by the index $\alpha + 1/2$, instead of 
the self-similarity index $H$. Hence, the parameter $\gamma$ 
does not affect the rates of uniform and local oscillations of $X$, and it only affects the constant in the limit theorems.
We further described explicitly the location dependence of the local oscillation properties of $X$ at a given point $t_0>0$. 
For example,  Theorems 1.4-1.6 of \cite{WX2021a} demonstrate  that the oscillations decrease
at the rate $t_0^{-\gamma/2}$ as $t_0$ increases in the 
contexts of Chung's LIL, the local modulus of continuity
and the tangent process,  respectively.  These properties are 
significantly different from those of fractional Brownian motion.

In the present paper, we continue the line of research in 
\cite{IPT2021, IPT2020b, WX2021a, WX2021b} and study the 
local oscillation property of the GFBM $X$ in the context of
Strassen's local LIL. 

Following Strassen's functional LIL for Brownian motion
\cite{Strassen},  such functional limit theorems have been established for more general Gaussian processes, in both limsup and liminf forms. 
Extensive contributions include, among others, the works  \cite{Arc95, 
Gantert, HL2005,IPT2021, KLT, LWH2006, MR1995, Oodaira72, 
Oodaira73, Revesz, Wang2005, WSX2020}. For further historical
background and references on functional LILs for some Gaussian 
and non‑Gaussian processes (such as the Hermite processes), we 
refer to Taqqu and Czado \cite{TC1985} and Li and Shao
\cite{LS2001}. However, except  \cite{Gantert, WSX2020}, all the aforementioned references were concerned with functional
LILs at infinity,  that is, the functional limits of an 
appropriately normalized sequence of the processes 
$\{X(nt)\}_{0\le t \le 1}$, as $n \to \infty$.
  
In this paper,  we study the limsup-type functional LIL for 
the GFBM $X$ at a given point $t_0>0$,  {when
$\gamma\in (0,1)$ and 
$\alpha\in (-1/2+\gamma/2, 1/2)$}. Specifically, we study the functional limit
behavior, as   $h \to 0$,  of the  normalized family
 \begin{align}\label{eq U}
  \mathcal U_{t_0, h}(x):=\frac{X(t_0+hx)-X(t_0)}{t_0^{-\gamma/2} |h|^{\alpha+1/2}\left(  {\log\log}(1/|h|)\right)^{1/2}}
  \ \ \ \ \text{for  } x\in [0,1].
   \end{align}
This yields a local version of Strassen's LIL for the GFBM $X$.

We now make a few remarks concerning \eqref{eq U} and Strassen's LIL established in \cite{IPT2021,WSX2020}. 
\begin{itemize}
\item Ichiba, Pang and Taqqu \cite{IPT2021} established a Strassen's LIL for the sequence of processes
$$
\frac{X(nx)}{ n^{H}\sqrt{ 2  {\log\log}\, n}}
  \ \ \ \ \text{for  }  x\in [0,1],
$$
as $n \to \infty$, where $H$ is defined in \eqref{Eq:H}. 
Hence,  the normalizing factor in \eqref{eq U} differs from 
that  in \cite{IPT2021}.

\item  Strassen's local LIL at the origin can be derived from the functional LIL of Ichiba, Pang and Taqqu \cite[Theorem 5.1]{IPT2021}, combined with
 the time-inversion property of the GFBM $X$  given in Proposition 7.1 of \cite{WX2021a}
and Gantert's time-inversion argument  \cite{Gantert}.
More precisely,  one first extends  Strassen's LIL for $X$ at infinity in \cite[Theorem 5.1]{IPT2021}  to the   
whole time axis  (see \cite[Theorem 1.4.1]{DS1989}); then,    applying Proposition 7.1 of \cite{WX2021a} and Gantert's time-inversion 
  argument yields the corresponding local LIL at the origin.

\item Determining  the functional limit of (\ref{eq U}) is related to the work of  Wang, Su and
Xiao \cite{WSX2020}, where sufficient conditions for Strassen's global and local LILs are established for a class of Gaussian random fields with stationary increments.  A key objective of \cite{WSX2020} is to clarify  how the long-range dependence index and the fractal index affect Strassen’s global and local LILs, respectively.

In \cite{WSX2020}, the stationarity of increments plays a crucial  role: it provides a stochastic spectral representation for the Gaussian random field that is used to create independence 
in the proofs, and it  guarantees that Strassen's local  LILs  are independent of location. Because the GFBM $X$ does not have stationary increments, the arguments in \cite{WSX2020} 
can not be carried over directly. In order to introduce independence,  we follow  \cite{TX2007} and \cite{WX2021a} by employing the integral representation of the Lamperti transformation 
$U = \{U(t)\}_{t \in \mathbb R}$ of the generalized Riemann-Liouville FBM  $Z$ (see (\ref{eq decom}) below). Certain estimates for the spectral density of $U$ obtained in \cite{WX2021a} 
also play an important role in the present paper.
\end{itemize}

Our main result,  Theorem \ref{thm local LIL} below,  gives an explicit description of how  Strassen's local LIL depends on the parameter pair $(\alpha, \gamma)$ and the location 
$t_0\in (0, \infty)$. This theorem strengthens the results obtained in \cite{WX2021a} and, together with Strassen's LIL for $X$ at infinity established in \cite{IPT2021}, completes the functional LIL theory for the GFBM $X$.

We also remark that Strassen's LILs have   been established for
empirical processes   \cite{AT89, F71, James75, Wellner},  
quantile processes \cite{D92},    increments of empirical and 
quantile processes   \cite{DM92}, empirical cumulative 
quantile regression \cite{RZ96},   {and} U-statistics  \cite{AG95}.  {For further 
information, we refer to \cite{LS2001}.}  In all these references, the underlying samples are assumed to be i.i.d. or 
at least stationary, and the corresponding Strassen's laws of 
the iterated logarithm are formulated in terms of the 
reproducing kernel Hilbert space of the scaling limit of the
random samples. Our results may therefore be useful for characterizing the limiting behavior of increments of empirical and quantile processes of non‑stationary time series whose scaling limit is a GFBM. This direction merits further investigation.

The remainder of the paper is organized as follows. In Section 2, we state the main theorems and their applications. We prove the main theorems in Section 3 by using the tools of the large deviation principle, a Fernique-type inequality, and   Lamperti's transformation.

\section{Statement of main theorem}
      
  Let $C([0,1])$  denote the space of  all continuous $\mathbb R$-valued functions on $[0,1]$ equipped with the
     sup-norm $$\|f\|_{\infty}:=\sup_{x\in[0,1]}|f(x)|.$$
             For any $t_0>0$ and $x_1, x_2\ge0$,  define
\begin{align}\label{eq rho}
\rho_{t_0}(x_1, x_2):=\lim_{h\rightarrow 0}\frac{\mathbb E\big[\left(X(t_0+hx_1)-X(t_0)\right)
\left(X(t_0+hx_2)-X(t_0)\right)\big]}{|h|^{ {2\alpha+1}}}.
\end{align}
When $\alpha\in (-1/2+\gamma/2, 1/2)$, a straightforward computation (cf.  \cite[(3.33)]{WX2021a}) yields
   \begin{align}\label{eq limit rho}
   \rho_{t_0}(x_1, x_2)=\frac{c_{2,1}}{2t_0^{\gamma}}\left(x_1^{2\alpha+1} +x_2^{2\alpha+1}-|x_1-x_2|^{2\alpha+1} \right),
   \end{align}
   where  
   \begin{align}\label{eq c21} 
  {  c_{2,1}=\frac{1}{2\alpha+1}+\int_0^{\infty}\left[(1+u)^{\alpha}-u^{\alpha}\right]^2du. }
   \end{align}
    Note that $\rho_{t_0}$ is the covariance function of $c_{2,1}^{1/2} t_{0}^{-\gamma/2} B^{\alpha+1/2}$, where
   $B^{\alpha+1/2}$ denotes  a FBM with index $\alpha+1/2$ on some probability space $(\Omega, \mathcal F, \mathbb P)$.

   Let $\mathcal L$ be the closed linear subspace of $L^2(\Omega)$  generated by  {$\left\{B^{\alpha+1/2}(t)\right\}_{t \ge 0}$}. 
   Define the map  $\psi:\mathcal L\rightarrow C([0,1])$ by
      \begin{align*}
      \psi(\xi)(t):=\mathbb E\left[B^{\alpha+1/2}(t)\xi \right].
      \end{align*}
      We consider the  Cameron-Martin space (also called the reproducing kernel Hilbert space)   induced by    $\rho_{t_0} $. It is the Hilbert space  
\begin{align}\label{eq RKHS}
   \mathcal H:=\{\psi(\xi);\ \xi\in \mathcal L\},
    \end{align}
   equipped  with the inner product 
 \begin{align}\label{eq RKHS inn}
 \langle \psi(\xi_1), \psi(\xi_2)\rangle_{\mathcal H} =\mathbb E\left[\xi_1\xi_2 \right].  
 \end{align}
   For more details, see \cite{Arc, Aro1950, DU99} and \cite[Chapter 6.1]{BHOZ}.  
   
 Define the rate function
      \begin{equation}\label{eq rate}
   I(z):= \left\{\begin{array}{ll}
 \inf\left\{\frac12 \mathbb E[\xi^2]; \, \xi\in \mathcal L, \psi(\xi)=z\right\}, \ \quad &\hbox{ if }\  z\in \mathcal H,\\
 +\infty,  \ \quad &\hbox{ otherwise,}
\end{array}
\right.
 \end{equation}
and  let  
 \begin{align}\label{eq unit ball}
 \mathcal S:=\big\{z\in C([0,1]);\ I(z)\le 1\big\}.
 \end{align}

  Let $X$ be the GFBM defined in  \eqref{eq X};  for notational convenience,  we extend it by setting $X(t):=0$ if $t<0$.
  For any $t_0>0$ and $h\neq 0$,  recall that $ \mathcal U_{t_0, h}= \{\mathcal U_{t_0, h}(x)\}_{x\in [0,1]}$ denotes  the Gaussian process given  by \eqref{eq U}.
    
The following theorem is the main result of the paper.
   \begin{theorem}\label{thm local LIL}   Assume $\gamma\in (0,1)$ and $\alpha\in(-1/2+\gamma/2, 1/2)$. Then   for every  $t_0>0$,
   with probability one,  the family $\left\{\mathcal U_{t_0, h}(x)\right\}_{x\in [0,1]}$  
   is relatively compact on $C([0,1])$ as $h\rightarrow 0$, and its set of   limit points is $\mathcal S$. Namely, for every $t_0\in
   (0,\infty)$,  
    \begin{align}\label{eq equiv1}
 \lim_{h\rightarrow 0}\inf_{f\in \mathcal S} \|\mathcal U_{t_0, h}-f \|_{\infty}=0 \ \ \ \ \text{a.s.},
 \end{align}
 and for each $f\in \mathcal S$,
    \begin{align}\label{eq equiv2}
 \liminf_{h\rightarrow 0} \|\mathcal U_{t_0, h}-f \|_{\infty}=0 \ \ \ \ \text{a.s.}
 \end{align}  
  In particular,  for every continuous functional $F: C([0,1])\rightarrow \mathbb R$,
 \begin{equation}\label{eq continuity}
 \mathbb P\left(\limsup_{h\rightarrow0} F(\mathcal U_{t_0, h})=\sup_{\xi\in \mathcal S}F(\xi) \right)=1.
 \end{equation}
 \end{theorem}

 Let us present some applications of Theorem \ref{thm local LIL},  which are inspired by   \cite{AM05, Gantert}. 
  {The derivation of the explicit constants in \eqref{eq e1}–\eqref{eq e4} relies on Lemma \ref{lem const}  below. }
 \begin{example}\label{exa LIL}
 \begin{itemize}
\item[(a)] \label{example usual LIL}
    Choosing  $F_1(\xi)=\xi(1)$,   \eqref{eq continuity} implies the usual LIL at a given point $t_0 > 0$:  
   \begin{align}\label{eq e1}
   \limsup_{h\rightarrow0  }  \frac{|X(t_0+h)-X(t_0)|}{ t_0^{-\gamma/2} h^{\alpha+1/2}   \left(  {2\log\log}(1/ |h|)\right)^{1/2}}
   = {\sqrt{ c_{2,1} } }\ \ \text{a.s.}
       \end{align}
       Here and below, $c_{2,1}$ is the constant from \eqref{eq c21}.
       
We remark that  the   LIL for the GFBM $X$ at the origin has already been established in \cite[Theorem 6.1]{IPT2021}
  and \cite[Proposition 7.1]{WX2021a}.

 \item[(b)]
    Choosing  $F_2(\xi)=\sup_{0\le t\le 1}|\xi(t)|$,    \eqref{eq continuity} implies the  local modulus of continuity: 
   \begin{align} \label{eq e2}
   \limsup_{h\rightarrow0+  }\sup_{|r|\le h} \frac{|X(t_0+r)-X(t_0)|}{ t_0^{-\gamma/2} h^{\alpha+1/2}     \left(  {2\log\log}(1/ h)\right)^{1/2}}
   = {\sqrt{c_{2,1}  }} \ \ \text{a.s.} 
     \end{align}
      {This result was previously  established in   \cite[Theorem 1.5(a)]{WX2021a} but without  the explicit constant; here we determine the constant.  } 

\item[(c)]  Let  $0<\delta<1$. Then,   choosing  $F_3(\xi)=\sup_{0\le t\le 1-\delta}|\xi(t+\delta)-\xi(t)|$,     \eqref{eq continuity} implies that 
        \begin{align}\label{eq e3}
     \limsup_{h\rightarrow0+}\sup_{0\le t\le h(1-\delta)}\frac{|X(t_0+t+h\delta )-X(t_0+t)|}{ t_0^{-\gamma/2}h^{\alpha+1/2}   \left( {2\log\log}(1/ h)\right)^{1/2}}
       {= c_{2,2}\sqrt{c_{2,1}}  } \ \ \ \text{a.s {.}}
        \end{align}
        {where $c_{2,2} \in \left[\delta^H,  \sqrt{2} \delta^H\right]$}. 
        {When $H=\frac12$,  we have $c_{2,2}=\delta^{\frac12}$.}
          \item[(d)]  Let  $0<\delta<1$. Then,  choosing  $F_4(\xi)=\sup_{0\le t\le 1-\delta}\sup_{0\le s\le \delta}|\xi(t+s)-\xi(t)|$,  \eqref{eq continuity} implies  that
        \begin{align}\label{eq e4}
        \limsup_{h\rightarrow0+}\sup_{0\le t\le h(1-\delta)}\sup_{0\le s\le h\delta}\frac{|X(t_0+t+s )-X(t_0+t)|}{t_0^{-\gamma/2} h^{\alpha+1/2}  
        \left(  {2\log\log}(1/ h)\right)^{1/2}}   {=c_{2,2}\sqrt{c_{2,1}}  } \ \ \ \text{a.s {.},}
        \end{align}
 {where $c_{2,2}$ is  the constant appearing in \eqref{eq e3}.} 
   \end{itemize}   
                  
  The last two cases  (c) and (d) in  Example \ref{exa LIL} are the ``small time statements", which correspond to Corollary 1.2.2
  in Cs\"{o}rg\H{o}  and R\'ev\'esz \cite{CR} for Brownian motion. 
\end{example}

  {
We now turn to the case $\gamma\in (0,1)$ and $\alpha\in (1/2, 1/2+\gamma/2)$. In this regime,    
Theorem 4.1 of \cite{IPT2021} (or  Theorem 1.1 of \cite{WX2021a}) shows that the sample path of $X$ is continuously differentiable,  
with derivative given by}
\begin{equation*}
\begin{split}
  {X'(t)=\alpha \int_{-\infty}^t (t-u)^{\alpha-1}|u|^{-\gamma/2}B(du). }
\end{split}
\end{equation*} 
         { 
  For any $t_0>0$ and $h\neq 0$,  define the Gaussian process 
   $\widetilde{\mathcal U}_{t_0, h}= \left\{\widetilde{\mathcal U}_{t_0, h}(x)\right\}_{x\in [0,1]}$ as} 
       \begin{align*}\label{eq U'}  {
 \widetilde{\mathcal U}_{t_0, h}(x):=\frac{X'(t_0+hx)-X'(t_0)}{\alpha t_0^{-\gamma/2} |h|^{\alpha-1/2}\left(\log\log(1/|h|)\right)^{1/2}}
  \ \ \ \ \text{for  } x\in [0,1].}
   \end{align*}
           
   \begin{theorem}\label{thm local LIL'}     {Assume $\gamma\in (0,1)$ and $\alpha\in(1/2, 1/2+\gamma/2)$. Then,   for every  $t_0>0$,
   with probability one,  the family $\left\{\widetilde{\mathcal U}_{t_0, h}(x)\right\}_{x\in [0,1]}$  
   is relatively compact on $C([0,1])$  as $h\rightarrow 0$, and its  set of   limit points is $\widetilde{\mathcal S}$.   Here $\widetilde{\mathcal S}$ 
   is defined as in \eqref{eq unit ball}, but with the index $\alpha+1/2$ replaced by $\alpha-1/2$.}
 \end{theorem}
   { 
The case $\alpha = 1/2$ is critical for the sample paths of the GFBM $X$. It can be proven that, when $\alpha = 1/2$, the sample path of $X$ 
is not differentiable, and the law of the iterated logarithm at any $t>0$ is established in \cite[Theorem 1.5]{WX2021a}. However, the problems on  
the exact uniform modulus of continuity and Chung's law of the iterated logarithm have remained open; see \cite[Theorems 1.1]{WX2021a} 
for an upper bound for the uniform modulus of continuity. We remark that the critical case of $\alpha = 1/2$ is excluded from the scope of the 
present work because Proposition \ref{prop LDP}, Lemmas  \ref{Lem: Fernique} and \ref{lem spectral} 
fail. It is an open problem to establish Strassen's local LIL for $X$ in this case.
}

\section{Proofs of the main results}
 
\subsection{Reduction to  the generalized Riemann-Liouville FBM}
 {We start with a decomposition of the GFBM $X$:}
 \begin{equation}\label{eq decom}
 \begin{split}
  X(t) &=\int_{-\infty}^0  \big((t-u)^{\alpha}-(-u)^{\alpha} \big) (-u)^{-\gamma/2} B(du)  +\int_0^t  (t-u)^{\alpha}
  u^{-\gamma/2} B(du)\\
  &=: Y(t) + Z(t).
\end{split}
\end{equation}
The Gaussian processes   $Y=\big\{Y(t)\big\}_{t\ge0}$ and  $Z=\big\{Z(t)\big\}_{ t\ge0}$ are independent. The
process $Z$, which is well-defined when $\alpha>-\frac12$ and $\gamma<1$,  is called  a {\it generalized Riemann-Liouville FBM} by Ichiba, Pang and Taqqu \cite{IPT2021}.

According to  Proposition 2.1 of \cite{WX2021a},  the process $Y$  admits a modification that is infinitely continuously differentiable on $(0,\infty)$. Hence, to  prove Theorems \ref{thm local LIL} and \ref{thm local LIL'}, it suffices to establish the corresponding statements for 
  the generalized Riemann–Liouville FBM $Z$ in place  of the  GFBM $X$   appearing in \eqref{eq U}.  These results are  presented  in Theorem \ref{thm local LILZ} and Corollary \ref{coro local LILZ'}  below. 

 {The sample-path properties of $Z$, such as   moment estimates, the exact uniform modulus of continuity, small ball probabilities, and Chung’s laws of the iterated logarithm, are established in \cite[Sections 3–6]{WX2021a}. }

 We extend $Z$ by setting  $Z(t):=0$ if $t< 0$.    For any $t_0>0$ and $h\neq 0$,   define
          \begin{align}\label{eq UZ}
   U_{t_0, h}(x):=\frac{Z(t_0+hx)-Z(t_0)}{t_0^{-\gamma/2}|h|^{\alpha+1/2}\left(  {\log\log}(1/|h|)\right)^{1/2}} \ \ \ \text{for }    x\in [0,1].
   \end{align}

   \begin{theorem}\label{thm local LILZ}      {Assume $\gamma\in (0,1)$ and $\alpha\in(-1/2, 1/2)$}. Then   for every  $t_0>0$,
   with probability one, the family $\left\{U_{t_0, h}
   (x)\right\}_{x\in [0,1]} $ defined by \eqref{eq UZ}  is 
   relatively compact on $C([0,1])$ as  $h\rightarrow 0 $, and
   its set of limit points is $\mathcal S$ defined by 
   \eqref{eq unit ball}. Namely, for every $t_0>0$,  
    \begin{align}\label{eq equiv1b}
 \lim_{h\rightarrow 0 }\inf_{f\in \mathcal S} \|U_{t_0, h}-f \|_{\infty}=0 \ \ \ \ \text{a.s.},
 \end{align}
 and for each $f\in \mathcal S$,
    \begin{align}\label{eq equiv2b}
 \liminf_{h\rightarrow 0 } \|U_{t_0, h}-f \|_{\infty}=0 \ \ \ \ \text{a.s.}
 \end{align}
  In particular,  for every continuous functional $F: C([0,1])\rightarrow \mathbb R$,
 \begin{equation}\label{eq continuity2}
 \mathbb P\bigg(\limsup_{h\rightarrow0} F(U_{t_0, h})=\sup_{\xi\in \mathcal S}F(\xi) \bigg)=1.
 \end{equation}
   \end{theorem}

  When $\alpha\in (1/2, 3/2)$, the derivative of $Z$ is  
      \begin{equation*}
\begin{split}
  {Z'(t)=\alpha \int_{0}^t (t-u)^{\alpha-1}|u|^{-\gamma/2}B(du). }
\end{split}
\end{equation*} 
In this case,  the derivative process $\{Z'(t)/\alpha\}_{t\ge0}$ is a generalized Riemann–Liouville FBM with indices $\gamma\in (0,1)$ and 
$\alpha-1\in (-1/2,   1/2)$.  See \cite[Lemma 3.3]{WX2021a}.
  
  Let
 \begin{align}\label{eq UZ'}
   \widetilde U_{t_0, h}(x):=\frac{Z'(t_0+hx)-Z'(t_0)}{\alpha t_0^{-\gamma/2}|h|^{\alpha-1/2}\left(  {\log\log}(1/|h|)\right)^{1/2}} \ \ \ \text{for }    
 t_0>0, h\neq 0,  x\in [0,1].
   \end{align}
   
According to Theorem \ref{thm local LILZ}, we have   the following  corollary. 
       \begin{corollary}\label{coro local LILZ'}      {Assume $\gamma\in (0,1)$ and $\alpha\in(1/2, 3/2)$}. Then   for every  $t_0>0$,
   with probability one, the family $\left\{\widetilde U_{t_0, h}(x)\right\}_{x\in [0,1]}$ defined by \eqref{eq UZ'}
     is relatively compact on $C([0,1])$ as  $h\rightarrow 0 $, and its set of   limit points is $\widetilde{\mathcal S}$, which is   given by \eqref{eq unit ball} with  $\alpha+1/2$ replaced by $\alpha-1/2$.
   \end{corollary}

 We prove Theorem \ref{thm local LILZ}   in the following subsections.

\subsection{Some lemmas}
 By \cite[Lemma  3.1]{WX2021a},   {when $\gamma\in (0,1)$ and  $\alpha\in (-1/2, 1/2)$,}
there exist positive constants $c_{3, 1}$ and  $c_{3,2}$ such that   for all $0 < s < t$,
 \begin{align}\label{Eq: Zmoment1}
c_{3,1}\frac{|t-s|^{2\alpha+1}}{t^{ \gamma}}\le     \mathbb E\left[\big(Z(t)-Z(s) \big)^2\right]
\le c_{3,2}\frac{|t-s|^{2\alpha+1}}{s^{ \gamma}}.
 \end{align} 

 \subsubsection{Large deviation principle for $\{U_{t_0, h_n}\}_{n\ge1}$}

   To prove Theorem \ref{thm local LILZ},  several technical lemmas are needed.  The following large deviation
   principle can be obtained by applying Theorem 5.2 of Arcones \cite{Arc}.

   \begin{proposition}\label{prop LDP} Assume that  {$\gamma\in (0, 1)$ and $\alpha\in (-1/2,1/2)$}.  
   Let $\{h_n\}_{n\ge1}$ be a sequence of  real numbers  {such that $h_n\rightarrow0 $  as $n\rightarrow \infty$}, and set $\beta_n=  {\log\log} (1/|h_n|)$. Then  for every   $t_0>0$, the sequence $\{U_{t_0, h_n}(x); x\in [0,1]\}_{n\ge1}$ satisfies  the large deviation principle on $C([0,1])$ with the speed $\beta_n$ and  the rate function $I$ defined  in \eqref{eq rate}; i.e.,   the following   hold:
    \begin{itemize}
    \item[(i).] For every $r\ge 0$, the set $\{z\in C([0,1]); \ I(z)\le r\}$ is compact in $C([0,1])$.

    \item[(ii).] For every closed set $\mathscr A\subset C([0,1])$,
    \begin{equation}\label{eq ULD}
    \limsup_{n\rightarrow\infty} \beta_n^{-1}\log \mathbb P\big(U_{t_0, h_n}\in \mathscr A \big)\le -\inf_{z\in \mathscr A}I(z).
    \end{equation}

    \item[(iii).] For every open set $\mathscr B\subset C([0,1])$,
    \begin{equation}\label{eq LLD}
    \liminf_{n\rightarrow\infty} \beta_n^{-1}\log \mathbb P\big(U_{t_0, h_n}\in \mathscr B \big)\ge -\inf_{z\in \mathscr B}I(z).
    \end{equation}
    \end{itemize}
   \end{proposition}

   To prove Proposition \ref{prop LDP}, we make use of   the following lemma due to Talagrand \cite{Tal95}.
     \begin{lemma}(\cite[Lemma 2.1]{Tal95})\label{Lem:Tail}
Let $\mathcal X = \{\mathcal X(t), t \in \mathbb R\}$ be a centered real-valued Gaussian process, 
and let $S \subset \mathbb R$ be a closed set equipped with the canonical metric  
$$d_{\mathcal X}(s, t) = \left[\mathbb E\big(\mathcal X(s) -\mathcal  X(t)\big)^2\right]^{1/2}.$$
Then there exists a  constant  $c_{3,3}>0$ such that for every $u >0$, 
\begin{equation*}
\mathbb P\Biggl\{ \sup_{s, t \in S} |\mathcal X(s) - \mathcal X(t)| \ge c_{3,3}\,
\bigg(u + \int_0^D \sqrt{\log N(d_{\mathcal X}, S, \varepsilon)}\, d\varepsilon \bigg)\Bigg\}
\le \exp\bigg( - \frac{u^2} {D^2}\bigg),
\end{equation*}
where $N(d_{\mathcal X}, S, \varepsilon)$ denotes the minimal number of open $d_{\mathcal X}$-balls
of radius $\varepsilon$ needed to cover $S$, and   $D= \sup\{ d_{\mathcal X}(s,t):\,
s,\, t \in S\}$ is the diameter of $S$.
\end{lemma}
 {Note. There is a  typographical error in    Lemma 2.1 of \cite{Tal95}, where     $D$  should be  $D^2$.     
A corrected statement and its justification are provided in Lemma 3.1 of \cite{DMX17}.}

\begin{proof}[Proof of Proposition \ref{prop LDP}] We prove this proposition   by  verifying Conditions (a.1)--(a.4) of Theorem 5.2 in   \cite{Arc}.
For convenience, we  only treat  the  case  $h\in (0,1]$; the case  $h\in [-1,0)$ is analogous. 

  For every   $t_0>0$ and   $h\in (0,1]$, define
  \begin{equation*}
V_{t_0,h}(x) :=\frac{ Z(t_0+hx)-Z(t_0)}{h^{\alpha+1/2}}, \qquad  x\ge0.
 \end{equation*} 
By  (3.33) of   \cite{WX2021a},  we have that for any $x_1, x_2\in [0, \infty)$,
 \begin{align*}
  \lim_{h \rightarrow0}  \mathbb E\big[V_{t_0,h }(x_1) V_{t_0, {h}}(x_2)   \big] = c_{3,4} t_0^{-\gamma}
  \left(x_1^{2\alpha+1}+x_2^{2\alpha+1}-|x_1-x_2|^{2\alpha+1}  \right),
   \end{align*}
   and
 \begin{align*}
 \lim_{h \rightarrow0}    \mathbb E\Big[ \big(V_{t_0, h }(x_1)- V_{t_0, h }(x_2)\big)^2\Big]=2 c_{3,4}  t_0^{-\gamma} \left|x_1-x_2\right|^{2\alpha+1},
   \end{align*}
   where  $c_{3,4}>0$.
Thus,   Conditions (a.1) and (a.2) of Theorem 5.2 in  \cite{Arc} are satisfied.

Additionally,  by  Lemma 3.1 of \cite{WX2021a},  there exists a constant $c_{3,5}>0$ such that for every $t_0>0, h\in (0,1]$, and $x_1, x_2\in [0, \infty)$,
   \begin{align*}
   \mathbb E\left[ \left(V_{t_0, h}(x_1)- V_{t_0, h}(x_2)\right)^2\right]\le c_{3,5} t_0^{-\gamma} \left|x_1-x_2\right|^{2\alpha+1},
   \end{align*}
 which verifies   Condition  (a.4)  of  Theorem 5.2 in   \cite{Arc}.

It remains to verify Condition (a.3) of  Theorem 5.2 in   \cite{Arc}, i.e., 
\begin{align}\label{Eq: a3}
\sup_{x\in [0,1]}|U_{t_0, h_n}(x)| \stackrel{\mathbb P}{\longrightarrow} 0, \ \ \ \text{as } n\rightarrow\infty.
\end{align}
In the following,  we will prove \eqref{Eq: a3} by using   Lemma \ref{Lem:Tail}.

 For any   $h_n\in (0, t_0/2)$, set $S:=[t_0, t_0+h_n]$.  By \eqref{Eq: Zmoment1}, we obtain 
 $$d_Z(x_1, x_2)\le c_{3,6}t_0^{-\gamma/2}|x_1-x_2|^{\alpha+1/2}$$
  for some constant $c_{3,6}>0$. Hence,
\begin{align}\label{Eq: D3}
D\le c_{3,6} t_0^{-\gamma/2}h_n^{\alpha+1/2},
\end{align}
   and
$$
N(d_Z, S, \varepsilon)\le c_{3,6}^{\frac{1}{\alpha+1/2}} h_n t_0^{-\frac{\gamma}{2\alpha+1}} \varepsilon^{-\frac{1}{\alpha+1/2}}.
$$
Using the change of  variables $\varepsilon=c_{3,6} t_0^{-\gamma/2}(h_nv)^{\alpha+1/2}$, we have
\begin{equation}\label{eq: int Z}
\begin{split}
\int_0^D\sqrt{\log  N\left(d_Z, S, \varepsilon\right)  }d\varepsilon\le&\, \int_0^{c_{3,6} t_0^{-\gamma/2}h_n^{\alpha+1/2}}
\sqrt{\log\left( c_{3,6}^{\frac{1}{\alpha+1/2}}    t_0^{-\frac{\gamma}{2\alpha+1}} h_n \varepsilon^{-\frac{1}{\alpha+1/2}}\right)  } d\varepsilon\\
\le  &\, (\alpha+1/2)c_{3,6}t_0^{-\gamma/2}h_n^{\alpha+1/2}\int_0^{1} \sqrt{ \log( 1/v)}  v^{\alpha-1/2}dv\\
\le &\, c_{3,7} t_0^{-\gamma/2}h_n^{\alpha+1/2},
\end{split}
\end{equation}
 {where  $c_{3,7}>0$ and the last inequality holds because}
$$ {\int_0^{1} \sqrt{ \log\left(\frac{1}{v}\right)}  v^{\alpha-1/2}dv<+\infty, \ \ \text{for } \alpha \in \left(-\frac12, \frac12\right).}$$

Since $\beta_n=  {\log\log} (1/h_n)\rightarrow \infty$ as $n\rightarrow\infty$,  applying  Lemma \ref{Lem:Tail}   with \eqref{Eq: D3} and  \eqref{eq: int Z},
there exists a constant $c_{3,8}>0$ such that   for every   $\varepsilon>0$ and all sufficiently large  $n$,
\begin{align*} 
 \mathbb P\left(\sup_{x\in [0,1]}|U_{t_0, h_n}(x)|\ge \varepsilon\right)
=&\,\mathbb P\left(\sup_{x\in [0,1]} \frac{|Z(t_0+ {h_n}x)-Z(t_0)|}{  t_0^{-\gamma/2} h_n^{\alpha+1/2}}
\ge  \beta_n^{1/2} \varepsilon\right)\\
\le &\, {\mathbb P\left(\sup_{x\in [0,1]} \frac{|Z(t_0+ {h_n}x)-Z(t_0)|}{  t_0^{-\gamma/2} h_n^{\alpha+1/2}}
\ge  \frac{\beta_n^{1/2} \varepsilon}{2}+ \int_0^D\sqrt{\log  N(d_Z, S, \varepsilon) }d\varepsilon\right)}\\
\le&\,  \exp\left(-c_{3,8} \beta_n \varepsilon^2\right).  
\end{align*}
 This establishes \eqref{Eq: a3} and completes the proof of Proposition \ref{prop LDP}.
    \end{proof}

\subsubsection{A Fernique-type inequality}
To prove Theorem  \ref{thm local LILZ},  we require the 
following  Fernique-type inequality for the generalized 
Riemann-Liouville FBM $Z$. The argument   is similar to that 
of Lemma 2.4 of \cite{WSX2020};  we include it for completeness.

\begin{lemma}\label{Lem: Fernique}  {Assume that   $\gamma\in (0, 1)$ and $\alpha\in (-1/2,1/2)$}. 
There exist constants $u_0\ge1$ and $c_{3,9}>0$ such that for any $\delta\ge a>0$  and  $u\ge u_0$,
\begin{equation*}\label{eq Fernique}
\begin{split}
&\mathbb P\left(\sup_{t\in [0,  \delta]}\sup_{s\in [0,a]} \left|Z(t_0+t+s)-Z(t_0+t) \right|\ge u\left(\sqrt{\log(\delta/a)}
+\frac{1}{\sqrt{\log(\delta/a)}} \right) a ^{\alpha+1/2}  t_0^{-\gamma/2}\right)\\
&\le\,  \exp\left(-c_{3,9} u^2\right).
\end{split}
\end{equation*}
 \end{lemma}
   \begin{proof} Let    $\{\xi(t,s)\}_{t, s\ge0}$ be a Gaussian process defined by
   $$\xi(t,s):=Z(t_0+t+s)-Z(t_0+t).$$
  Let $S:=[0,\delta]\times[0,a]$. By  \eqref{Eq: Zmoment1},    we have 
   \begin{equation*}
   \begin{split}
    \, d_{\xi}\big((t_1, s_1), (t_2, s_2)\big)
   \le&\,   \left(\mathbb E\left[\big(Z(t_0+t_1+s_1)-Z(t_0+t_2+s_2)\big)^2 \right]\right)^{1/2}\\
   &\, \, +  \left(\mathbb E\left[\big(Z(t_0+t_1)- {Z}(t_0+t_2)\big)^2 \right]\right)^{1/2}\\
   \le &\,  2c_{3,2}^{\frac12}t_0^{-\gamma/2}\left(|t_1-t_2|^{\alpha+1/2}+|s_1-s_2|^{\alpha+1/2}\right),
   \end{split}
   \end{equation*} 
   where $c_{3,2}$ is the constant appeared in \eqref{Eq: Zmoment1}.
   
 Using \eqref{Eq: Zmoment1} again,    we have    {that for  any $s_1, s_2\in (0,a)$},
  \begin{equation}\label{eq D2}
   \begin{split}
   D=&\sup_{(t_i, s_i)\in S, i=1,2} \Bigg\{\mathbb E\Big( \big[Z(t_0+t_1+s_1)-Z(t_0+t_1)\big] 
     -\left[Z(t_0+t_2+s_2)-Z(t_0+t_2) \right] \Big)^2 \Bigg \}^{1/2}\\
   \le &\,       \left(\mathbb E\left[\left(Z(t_0+t_1+s_1)-Z(t_0+t_1)\right)^2 \right]\right)^{1/2} \\
&\,\,    + \left(\mathbb E\left[\left(Z(t_0+t_2+s_2)- {Z}(t_0+t_2)\right)^2 \right]\right)^{1/2}\\
   \le &\,   {2c_{3,2}^{\frac12}}t_0^{-\gamma/2} a^{\alpha+1/2},
   \end{split}
   \end{equation}
    and
   \begin{equation*}
   N(d_{\xi}, S,\varepsilon)\le \left( {2c_{3,2}^{\frac12}}\right)^{\frac{2}{\alpha+1/2}} t_0^{-\frac{\gamma}{ \alpha+1/2}}   \frac{a\delta}{\varepsilon^{2/(\alpha+1/2)}}.
   \end{equation*}
  Consequently,   
   \begin{equation}\label{eq int}
   \begin{split}
   &\int_0^D\sqrt{\log  N(d_{\xi}, S,\varepsilon)  } d\varepsilon\\
   \le&\,  \int_0^{ {2c_{3,2}^{\frac12}} t_0^{-\gamma/2}a^{\alpha+1/2}}\sqrt{ \log  \left( \left(  {2c_{3,2}^{\frac12}}\right)^{\frac{2}{\alpha+1/2}} t_0^{-\frac{\gamma}{ \alpha+1/2}}
   \frac{a\delta}{\varepsilon^{2/(\alpha+1/2)}}\right)  } d\varepsilon\\
   = &\,  {2c_{3,2}^{\frac12}} t_0^{-\gamma/2}\int_0^{ a }\sqrt{2 \log \left(   \frac{\sqrt{a\delta}}{ t }\right)  }  d t^{\alpha+1/2}\\
  =&\,  { 2c_{3,2}^{\frac12}} t_0^{-\gamma/2}\left(\sqrt{\log( \delta/a)} a^{\alpha+1/2}+ \int_0^{a} \frac{1}{\sqrt{2 \log \left(    \frac{\sqrt{a\delta}}{ t }\right)  }} t^{\alpha-1/2}dt \right)\\
    \le &\,  {2c_{3,2}^{\frac12} }t_0^{-\gamma/2}  \left(\sqrt{\log(\delta/a)}a^{\alpha+1/2}+   \frac{2}{\sqrt{\log( \delta/a )}} a^{\alpha+1/2}\int_0^{\infty} ue^{-(\alpha+1/2)u^2}du \right)\\
    \le&\,   {4c_{3,2}^{\frac12}  }
\cdot \int_0^{\infty} ue^{-(\alpha+1/2)u^2}du         \left(\sqrt{\log(\delta/a)}  + \frac{1}{\sqrt{\log(\delta/a)}}  \right)  a^{\alpha+1/2} t_0^{-\gamma/2} \\
=: &\,  c_{3,10}   \left(\sqrt{\log(\delta/a)}  + \frac{1}{\sqrt{\log(\delta/a)}}  \right)   a^{\alpha+1/2} t_0^{-\gamma/2}.
   \end{split}
   \end{equation}
 {Here,  $c_{3,10}>0$,   the first and second-last steps follow from the changes of variables   $\varepsilon= 2c_{3,2}^{\frac12} t_0^{-\gamma/2}t^{\alpha+1/2}$ 
and    $t=a e^{-u^2}$,  respectively.}
 
  By Lemma \ref{Lem:Tail}, \eqref{eq D2}, and  \eqref{eq int},   we have that  for  all $u\ge 2 c_{3,10}$,
      \begin{align*}
     &   {\mathbb P \left(\sup_{t\in [0,  \delta]}\sup_{s\in [0,a]} \left|Z(t_0+t+s)-Z(t_0+t) \right|\ge u\left(\sqrt{\log(\delta/a)}
+\frac{1}{\sqrt{\log(\delta/a)}} \right) a ^{\alpha+1/2}  t_0^{-\gamma/2}\right) }\\ 
  {\le }&\,   { \mathbb P\Bigg(\sup_{t\in [0,  \delta]}\sup_{s\in [0,a]} \left|Z(t_0+t+s)-Z(t_0+t) \right| }\\
 &\,\,\,\, \,\,\,\, \,\,\,\, \,\,\,\,   { \ge \frac{u}{2}\left(\sqrt{\log(\delta/a)}+\frac{1}{\sqrt{\log(\delta/a)}} \right) a ^{\alpha+1/2}  t_0^{-\gamma/2} +\int_0^D\sqrt{\log  N(d_{\xi}, S,\varepsilon)  } d\varepsilon
 \Bigg) }\\
  &  { \le\,  \exp\left(-c_{3,9} u^2\right), }
   \end{align*}   
   where $c_{3,9}>0$.   The proof is complete.
   \end{proof}

\subsubsection{Lamperti's transformation}\label{sec:Lamperti}  {To prove \eqref{eq equiv2b},
we  apply Lamperti's transformation, following the approach in   \cite{TX2007} and \cite{WX2021a}.}

Define the centered stationary Gaussian process $U =\{U (t)\}_{t\in\mathbb R}$ via  Lamperti's transformation of $Z$:
 \begin{align*}
 U(t):=e^{-tH} Z\left(e^t\right)  \ \ \ \text{for all } t\in \mathbb R.
 \end{align*}
 {Denote by   $r_U(t):=\mathbb E \big[U (0)U (t)\big]$   the covariance function.  A direct computation yields 
 \begin{align}\label{eq rt}
 r_U(t)=&\, e^{-tH}\int_0^{1\wedge e^t}\left(e^t-u\right)^{\alpha}(1-u)^{\alpha} u^{-\gamma}du \ \ \ \ \text{for all } t\in \mathbb R.
 \end{align}}
       \begin{lemma}\label{lem ru}  {
 The covariance function $r_U(t)$  is an even function and
    \begin{align}\label{eq rulim}
  \lim_{t\rightarrow +\infty} e^{t(1-\gamma)/2} r_U(t)   =\mathcal B(1+\alpha, 1-\gamma),
    \end{align}
    where $\mathcal B$ denotes the Beta function.    }
 \end{lemma}
  \begin{proof} 
    {
     For any $t>0$,  using \eqref{eq rt} and the change of variables $u=e^{-t}v$, we obtain}
    {     \begin{align*}   
 r_U(-t)=&\, e^{tH} \int_0^{e^{-t}} \left(e^{-t}-u\right)^{\alpha} (1-u)^{\alpha} u^{-\gamma}du \\
 =&\,  e^{tH} e^{-(2\alpha-\gamma+1)t} \int_0^{1} \left(1- v \right)^{\alpha}\left(e^{t}-v\right)^{\alpha}  v^{-\gamma} dv\\
 =&\, e^{-tH}  \int_0^{1}  \left(e^{t}-v\right)^{\alpha}  \left(1- v \right)^{\alpha} v^{-\gamma} dv\\
 =&\, r_U(t).
 \end{align*} }
       {Hence, $r_U$ is an even function.}  {Next, we prove \eqref{eq rulim}.}
      
          {By  \eqref{Eq:H}, \eqref{eq rt}, and the dominated convergence theorem, we have}
     {          \begin{align*}  
     \lim_{t\rightarrow +\infty}   e^{t(1-\gamma)/2}  r_U(t)  
     = &\,  \lim_{t\rightarrow +\infty}    \int_0^{1}  \left(1-e^{-t}u\right)^{\alpha}  \left(1- u \right)^{\alpha} u^{-\gamma} du\\
     = &   \mathcal B(1+\alpha, 1-\gamma).
      \end{align*}}
        {The proof is complete. }
  \end{proof}
  
By Bochner's theorem (see Lo\'eve \cite[Page 220]{Lo}), $r_U$
is the Fourier transform of a finite measure $F_U$, called the 
spectral measure of $U$. 
From Lemma \ref{lem ru}, we have  $r_U(\cdot)\in L^1(\mathbb R)$. Consequently,    $F_U$ has a continuous spectral
 density function $f_U$, which can be expressed as the inverse Fourier transform of $r_U(\cdot)$:
 $$
 f_U(\lambda)=\frac1\pi\int_0^{\infty}r_U(t)\cos(t\lambda)\,dt  \quad \hbox{ for all } \, \lambda \in \mathbb R.
 $$
  It is well known that the stationary Gaussian process $U $  admits the stochastic integral representation 
   \begin{align}\label{eq U int}
   U(t)=\int_{\mathbb R} e^{i\lambda t}\, W(d\lambda)    \ \ \ \ \text{for all } t\in \mathbb R,
   \end{align}
 where $W$ is a complex Gaussian random measure whose control measure is $F_U$.

 We recall the following estimates from Wang and Xiao \cite[Theorem 4.2]{WX2021a}.

 \begin{lemma}\cite[Theorem 4.2]{WX2021a}\label{lem spectral}
Assume that  $\gamma \in(0,1)$ and $\alpha\in (-1/2, 1/2)$. There exist   constants $u_0>0$ and
$c_{3,11}>0$   such that for any $u >u_0$,
 \begin{align} 
 &\int_{|\lambda|<u}\lambda^2 f_U(\lambda) d\lambda\le  \, c_{3,11} u^{1-2\alpha}, \label{Eq: sp1}\\
 &   \int_{|\lambda|\ge u}  f_U(\lambda) d\lambda\le  \,  c_{3,11} u^{-(2\alpha+1)}. \label{Eq: sp2}
 \end{align}
 \end{lemma}

\subsection{Proof of Theorem \ref{thm local LILZ}}  
To prove Theorem \ref{thm local LILZ}, it suffices to establish \eqref{eq equiv1b} and \eqref{eq equiv2b}, which will be proved separately below.

   \begin{proof}[Proof of   \eqref{eq equiv1b}]
For convenience, we only consider the case   $h\in [0,1]$, the case   $h\in [-1,0]$ is analogous.
Let $1<\theta<2$ be a constant to be chosen later. For any $k\in \mathbb Z_+$,    let  $h_k:= \theta^{- k}$ and  
$$
A_k:=\left\{h\in [0,1]: \, \theta^{-(k+1)}\le h\le \theta^{-k} \right\}.
$$
 {Recall $\mathcal S$ defined by \eqref{eq unit ball}.}
By \eqref{eq ULD},  for all $\varepsilon>0$ and $\mu>0$, there exists a constant $k_0>0$ such that for all $k\ge k_0$,
\begin{equation}
\begin{split}\label{Eq:3.24}
\mathbb P\left(\inf_{f\in  {\mathcal S}}\|U_{t_0, h_k}-f\|_{\infty}\ge \varepsilon  \right)
 = &\,\mathbb P\big( U_{t_0, h_k}\in ( {\mathcal S^{\varepsilon}})^c \big)\\
 \le & \, c_{3,12}\left[k\log(\theta)\right]^{-(1-\mu)\inf_{z\in ( {\mathcal S^{\varepsilon}})^c}  I (z) },
\end{split}
\end{equation}
where $c_{3,12}>0$, $\left(\mathcal S^{\varepsilon}\right)^c$ denotes the complement of the set $\mathcal S^{\varepsilon}$, and
$$
\mathcal S^{\varepsilon}:=\left\{g\in C([0,1]):\, \inf_{f\in \mathcal  S}\|f-g\|_{\infty}<\varepsilon \right\}.
$$
Choosing  $\mu$ small enough so that
$$
\eta:=(1-\mu)\inf_{z \in ( {\mathcal S^{\varepsilon}})^c}  I (z)>1, 
$$
the last terms in \eqref{Eq:3.24} are summable in $k$. By the Borel-Cantelli  lemma, we have
\begin{equation}\label{eq: upper1}
L_1:=\limsup_{k\rightarrow\infty}\inf_{f\in \mathcal S}\|U_{t_0, h_k}-f\|_{\infty}=0, \ \ \ \text{a.s.}
\end{equation}

Applying  Lemma \ref{Lem: Fernique} with $\delta=\theta^{-k}$ and $a=(1-\theta^{-1})\theta^{-k}$, we have that for every $\varepsilon>0$,
\begin{equation}\label{eq: upper2}
\begin{split}
&\,\mathbb P\left(\sup_{t\in [0,\theta^{-k}]}\sup_{s\in [0,(1-\theta^{-1})\theta^{-k}]}
 \frac{|Z(t_0+t+s)-Z(t_0+t)|}{  {t_0^{-\gamma/2}} \theta^{-k(\alpha+1/2)}    \sqrt{\log\log \theta^{k(\alpha+1/2)}}}>\varepsilon\right)\\
\le &\,  \exp\left( - c_{3,13} \varepsilon^2 c(\theta)   \log\log \theta^{k(\alpha+1/2)}\right),
\end{split}
\end{equation}
where  $c_{3,13}>0$ and $$c(\theta):=\frac{ (1-\theta^{-1})^{- {(2\alpha+1)}}  }{\left(\sqrt{\log\frac{\theta}{\theta-1}}
+ 1/{\sqrt{\log\frac{\theta}{\theta-1}}}\right)^2}.$$
Taking  $1<\theta<2$  sufficiently  close to $1$ so that $c_{3,13}c(\theta)\varepsilon^2>1$,  the series on the left-hand side of \eqref{eq: upper2} is  summable in $k$. 
By the Borel-Cantelli lemma,
\begin{equation}\label{eq: upper3}
\begin{split}
L_2:=& \limsup_{k\rightarrow\infty}\sup_{h\in A_k} \|U_{t_0,h}-U_{t_0, h_k}\|_{\infty}\\
=&\,\limsup_{k\rightarrow\infty}\sup_{h\in A_k}\sup_{x\in [0,1]}\frac{|Z(t_0+hx)-Z(t_0+h_kx)|}{ {t_0^{-\gamma/2}}h^{\alpha+1/2}\sqrt{\log\log (1/h) }}\\
\le&\, \limsup_{k\rightarrow\infty}\sup_{t\in [0, \theta^{-k}]}\sup_{s\in [0,(1-\theta^{-1}) \theta^{-k}]}\frac{|Z(t_0+t+s)-
Z(t_0+t)|}{ {t_0^{-\gamma/2}}\theta^{-(k+1)(\alpha+1/2)} \sqrt{\log\log \theta^{ k(\alpha+1/2)}} }\\
= &\, 0, \ \ \ \text{a.s.}
\end{split}
\end{equation}
Combining \eqref{eq: upper1} and \eqref{eq: upper3} yields \eqref{eq equiv1b}.  The proof is complete. 
\end{proof}

 \begin{proof}[Proof of \eqref{eq equiv2b}]
 Let $f\in \mathcal S$ and $0<\tau<1$. For $n\ge1$, let $h_n=\exp\left(-n^{1+\tau}\right)$. To prove  \eqref{eq equiv2}, it suffices  to show that
\begin{equation}\label{eq equiv2'}
\liminf_{n\rightarrow\infty}\|U_{t_0, h_n}-f\|_{\infty}=0 \ \ \ \ \text{a.s.}
\end{equation}

By \eqref{eq LLD},  for every $\varepsilon>0$ and $\mu>0$, there exists   $n_0>0$ such that for all $n\ge n_0$,
\begin{equation*}
\begin{split}
\mathbb P\big( \|U_{t_0, h_n}-f\|_{\infty}< \varepsilon  \big)=&\,\mathbb P\big( U_{t_0, h_n}\in  B(f,\varepsilon) \big)\\
\ge &\,     c_{3,14}  n^{-(1+\tau)(1+\mu)\inf_{g\in  B(f,\varepsilon)}  I (g) },
\end{split}
\end{equation*}
where  $c_{3,14}>0$ and
$$
B(f, \varepsilon):=\Big\{g\in C([0,1]);\, \|g-f\|_{\infty}<\varepsilon \Big\}.
$$
Since $\inf_{g\in  B(f,\varepsilon)}  I (g)<1$, we can choose  $\mu$ and $\tau$ sufficiently small enough so that
$$
 (1+\tau)(1+\mu)\inf_{g\in  B(f,\varepsilon)}  I (g) \le 1.
$$
Consequently, 
\begin{align}\label{eq: lower-0}
\sum_{n=1}^{\infty}\mathbb P\left(\|U_{t_0, h_n}-f\|_{\infty}<\varepsilon \right)=+\infty.
\end{align}
Because the events in the summands of \eqref{eq: lower-0} are not independent, some additional care
is required to  complete   the proof.

Following the approach in  Tudor and Xiao \cite{TX2007} and Wang and Xiao \cite{WX2021a},
we will employ  the following stochastic integral representation of $Z$, which follows from the spectral representation \eqref{eq U int} of $U$:
\begin{equation}\label{Eq:Rep1}
Z(t) = t^{H} \int_{\mathbb R} e^{i \lambda \log t}\,
W(d\lambda),\ \ \ \ \ t>0.
\end{equation} 

For $n\ge 3$, let
$$
d_n:=\exp\left(n^{1+\tau}+n^{\tau} \right).
$$
Define two Gaussian processes $Z_n$ and $\widetilde{Z}_n$, respectively, by
\begin{align*}%
Z_n(t) =&\,  t^{H} \int_{|\lambda| \in (d_{n-1}, d_n]} e^{i \lambda\log  t}\, W(d\lambda),\\ 
  \widetilde{Z}_n(t) =&\,  t^{H} \int_{|\lambda| \notin (d_{n-1}, d_n]}  e^{i \lambda \log  t}\, W(d\lambda). 
\end{align*}
  Clearly, 
\begin{equation}\label{Eq:Xn3}
Z(t) = Z_n(t) + \widetilde{Z}_n(t) \ \ \text{for all  } t >0.  
\end{equation}
Observe that  the Gaussian processes $Z_n\, (n = 1, 2, \ldots)$ are
independent and, moreover, for each $n \ge 1$, the processes $Z_n$ and $\widetilde{Z}_n$ are independent as well.

For any $|s|<h_n$, we have 
\begin{equation}\label{Eq:J1}
\begin{split}
\mathbb E\Big(\big(\widetilde{Z}_n(t+s) - \widetilde{Z}_n(t)\big)^2\Big) &=
\int_{|\lambda|\le d_{n-1}} \big| (t+s)^{H} \, e^{i \lambda \log  (t+s)} -
t^{H} \, e^{i \lambda \log  t}\big|^2 \,f_U(\lambda)\, d\lambda\\
&\quad + \int_{|\lambda| > d_{n}} \big| (t+s)^{H} \, e^{i \lambda
\log  (t+s)} - t^{H} \, e^{i \lambda \log  t}\big|^2 \,f_U(\lambda)\,
d\lambda\\
&=: \mathcal J_1 + \mathcal J_2.
\end{split}
\end{equation}

The second term $\mathcal J_2$ is easy to estimate.  For any $|s|\le  {h_n}$,  by   \eqref{Eq: sp2}, there exists $c_{3,15}>0$ such  that
\begin{equation}\label{Eq:J2}
\begin{split}
\mathcal {J}_2 \le&\, 2\,\left(t^{2H}+ (t+s)^{2H}\right)\, \int _{|\lambda| >
d_{n}}\,f_U(\lambda)\, d\lambda\\
\le&\, c_{3,15}\,   \,
d_n^{-(2\alpha+1)}\\
=&\, c_{3,15}\,   \,
h_n^{2\alpha+1} \exp\left(-(2\alpha+1)n^{\tau}\right).
\end{split}
\end{equation} 

To bound the first term $\mathcal J_1$,   using the    elementary inequalities:
 \begin{align*}
\bigl|(x+y)^H-x^H\bigr| &\le |H| \max\left\{x^{H-1},\,(x+y)^{H-1}\right\}|y|, && (0\le |y|<x),\\ 
1-\cos x &\le x^{2}, && (x\in\mathbb R),\\ 
\log (1+x) &\le x, && (x\ge0),
\end{align*}
 we   obtain  that for any $|s|\le h_n$,
\begin{equation}\label{Eq:J3}
\begin{split}
\mathcal J_1 =&\, \int_{|\lambda|\le d_{n-1}} \left[\left((t+s)^{H} -
t^{H}\right)^2  + 2 (t+s)^{H}\,t^{H} \Big( 1 - \cos\Big( \lambda \log 
\frac{t+s}t\Big)\Big)\right] \,f_U(\lambda)\, d\lambda\\
  \le &\,   {H^2}\max\left\{ t^{2(H-1)},  (t+s)^{2(H-1)}  \right\}s^{2}  \int_{\mathbb R}
\,f_U(\lambda)\, d\lambda \\
& \quad \ + 2(t+s)^{H}\, t^H\, \log  ^2\left(1+\frac{s}t\right)\,
\int_{|\lambda|\le d_{n-1}}  \lambda^2
\,f_U(\lambda)\, d\lambda\\
  \le &\, c_{3,16}(t)  \left( \int_{\mathbb R}
\,f_U(\lambda)\, d\lambda+ \int_{|\lambda|\le d_{n-1}}  \lambda^2
\,f_U(\lambda)\, d\lambda \right)h_n^{2}.
\end{split}
\end{equation}
Here, the positive constant $c_{3,16}(t)$ depends on $t$ but is independent of $s$.

It follows from the continuity of $f_U$ and \eqref{Eq: sp2} 
that
\begin{equation}\label{Eq:J4}
\int_{\mathbb R} \,f_U(\lambda)\, d\lambda  <+\infty.
\end{equation}
Moreover, using  \eqref{Eq: sp1}, we  obtain that for all sufficiently large  $n$,
\begin{equation}\label{Eq:J5}
\begin{split}
& \left(\int_{|\lambda|\le d_{n-1}}  \lambda^2
\,f_U(\lambda)\, d\lambda \right)h_n^{2} \\ 
\le  &\, c_{3,11}  \exp\left((1-2\alpha)\left[(n-1)^{1+\tau}+(n-1)^{\tau}\right] \right)  \exp\left(-2n^{1+\tau} \right)\\
= &\,  c_{3,11}  \exp\left(-(1-2\alpha)\left[n^{1+\tau}-(n-1)^{1+\tau}-(n-1)^{\tau}\right] \right)  \exp\left(-(1+2\alpha)n^{1+\tau} \right)\\
\le &\,  c_{3,11}  \exp\left(-(1-2\alpha)\tau n^{\tau}/2 \right) \exp\left(-(1+2\alpha)n^{1+\tau} \right),
\end{split}
\end{equation} 
 {where the last inequality holds, since      }
$$ {
n^{1+\tau}-(n-1)^{1+\tau}-(n-1)^{\tau}>\frac{\tau n^{\tau}}{2}    \ \   \text{ for all sufficiently  large  } n. }
 $$
  Thus, by \eqref{Eq:J1}-\eqref{Eq:J5}, we obtain   for any $\alpha<1/2$ that 
\begin{align*}
\mathcal J_1 +\mathcal J_2\le c_{3,17} \exp\left(- c_{3,18}n^{\tau}\right) \exp\left(-(2\alpha+1)n^{1+\tau}\right),
\end{align*}
for some constants $c_{3,17}, c_{3,18}>0$. Therefore,   the diameter $D$ of $S:=[t_0, t_0+h_n]$ satisfies 
 $$D< c_{3,17}^{1/2} \exp\left(- c_{3,18} n^{\tau}/2 \right)h_n^{\alpha+1/2}, $$
  for  all sufficiently large $n$.

 {
By \eqref{Eq: Zmoment1}, \eqref{Eq:Xn3} and the independence of  $Z_n$ and $\widetilde{Z}_n$,    we have
  }
 \begin{align*} 
 {d_{\widetilde{Z}_n}(s,t)} & {= \left(\mathbb E\left[|\widetilde{Z}_n(t_0+s)-\widetilde{Z}_n(t_0+t)|^2\right]\right)^{\frac12}   } \\
& {\le \left(\mathbb E\left[| Z(t_0+s)- Z(t_0+t)|^2\right]\right)^{\frac12}} \\
& {\le c_{3, 2} t_0^{-\gamma/2}|t-s|^{\alpha+1/2},}
\end{align*}
 {where the constant  $c_{3,2}$ is the one appearing  in  \eqref{Eq: Zmoment1}.   Consequently, 
  $$
N(d_{\widetilde{Z}_n}, S, \varepsilon)\le  {c_{3,2}^{\frac{1}{\alpha+1/2}}} h_n \varepsilon^{-1/(\alpha+1/2)}t_0^{-\gamma/(2\alpha+1)}.
$$
}  Similarly to \eqref{eq int}, we obtain
$$
\int_0^D\sqrt{\log( N(d_{\widetilde{Z}_n}, S, \varepsilon))}d\varepsilon\le c_{3,19} \exp\big(-c_{3,20}  n^{\tau}\big)h_n^{\alpha+1/2},
$$
for some constants $c_{3,19}, c_{3,20}>0$. Hence, by Lemma \ref{Lem:Tail}, there  exist  $u_0, c_{3,21}, c_{3,22}>0$ such that for every $u\ge u_0$,
\begin{equation*}
\mathbb P\left(\sup_{x\in [0,1]}\left|\widetilde{Z}_n(t_0+h_nx)-\widetilde{Z}_n(t_0)\right|
\ge u\exp\left(-c_{3,21} n^{\tau}\right)h_n^{\alpha+1/2}\right)
\le  \exp\left(-c_{3,22} u^2\right).
\end{equation*}
This implies that for every $\varepsilon>0$,
\begin{equation}\label{eq: sum1}
\sum_{n=1}^{\infty} \mathbb P\left(\frac{\sup_{x\in [0,1]}\left|\widetilde{Z}_n(t_0+h_nx)-\widetilde{Z}_n(t_0)\right|   }
{ t_0^{-\gamma/2}h_n^{\alpha+1/2}\sqrt{\log\log( 1/h_n)}} \ge \varepsilon\right)<+\infty.
\end{equation}
Hence, by the Borel--Cantelli lemma, we have
\begin{equation}\label{eq: lim1}
I_1:=\limsup_{n\rightarrow\infty}\sup_{x\in [0,1]}\frac{ \left|\widetilde{Z}_n(t_0+h_nx)-\widetilde{Z}_n(t_0)\right|   }
{ t_0^{-\gamma/2}h_n^{\alpha+1/2}\sqrt{\log\log( 1/h_n)}}=0 \ \ \ \text{a.s.}
\end{equation}

It follows from \eqref{eq: lower-0} and \eqref{eq: sum1} that
\begin{align*}
&\sum_{n=1}^{\infty}\mathbb P\Bigg(\sup_{x\in [0,1]}\bigg|\frac{   Z_n(t_0+h_nx)- Z_n(t_0)   }{ t_0^{-\gamma/2}
h_n^{\alpha+1/2}\sqrt{\log\log (1/h_n)}}  -f(x)\bigg|  \le  2\varepsilon\Bigg)\\
 \ge &\,   \sum_{n=1}^{\infty}\mathbb P\big(\|U_{t_0, h_n}-f\|_{\infty}<\varepsilon \big)-
\sum_{n=1}^{\infty} \mathbb P\Bigg(\frac{\sup_{x\in [0,1]}\bigg|\widetilde{Z}_n(t_0+h_nx)-\widetilde{Z}_n(t_0)\bigg|   }
{ t_0^{-\gamma/2}h_n^{\alpha+1/2}\sqrt{\log\log( 1/h_n)}} \ge \varepsilon\Bigg)\\
 =\, & +\infty.
\end{align*}
Since the processes $\{Z_n\}_{n\ge3}$ are   independent,   the Borel--Cantelli lemma yields
\begin{equation}\label{eq: lim2}
I_2:=\liminf_{n\rightarrow\infty} \sup_{x\in [0,1]} \bigg|\frac{   Z_n(t_0+h_nx)- Z_n(t_0)  }
{ t_0^{-\gamma/2}h_n^{\alpha+1/2}\sqrt{\log\log (1/h_n)}}-f(x)\bigg|=0 \ \ \ \text{a.s.}
\end{equation}
 Combining  \eqref{eq: lim1} and \eqref{eq: lim2}, we  obtain
\begin{equation*}
 \liminf_{n\rightarrow\infty} \|U_{t_0, h_n}-f\|_{\infty}\le I_1+I_2=0\ \ \ \ \text{a.s.}
\end{equation*}
This verifies \eqref{eq equiv2'} and completes the proof of Theorem \ref{thm local LILZ}. 
    \end{proof}

\subsection{ {On the  explicit constants in Example \ref{exa LIL}} }\label{subsection RKHS} 
 {
Applying  \eqref{eq continuity} in Theorem \ref{thm local LIL}, we find that  the   constants in  Example \ref{exa LIL} are given,  
respectively, by
$$
\sup_{\xi\in \mathcal S}F_i(\xi), \ \ \ i=1, \cdots, 4.
$$  
The derivation of their exact values relies on a more explicit formula for the inner product \eqref{eq RKHS inn}.  To this end, we first 
recall from  \cite[Chapter 6.1]{BHOZ} and \cite{DU99}   some properties of the Cameron-Martin space for a fractional Brownian motion 
$\left\{B^H(t)\right\}_{t\in [0,1]}$ of index $H\in (0,1)$. }

 {
Recall  the  covariance function  of $B^H$ is }
 \begin{align}\label{eq cov 1}  {
R_H(t, s)=\mathbb E\left[B^H(t) B^H(s) \right]=\frac12\left(|t|^{2H}+|s|^{2H}-|t-s|^{2H}\right), \ \ \ s, t>0. }
 \end{align} {
 Let $K_H(t, s)$ be a deterministic kernel such that }
 \begin{align}\label{eq cov 2}  {
 R_H(t, s)=\int_0^{t\wedge s}K_H(t, u)K_H(s, u)du. }
 \end{align}  {
 When $H=\frac12$, }
 \begin{align}\label{eq KH1} {
 K_H(t, u)=\mathbf 1_{[0,t]}(u).  }
 \end{align} {
    When $H\neq \frac12$, for the explicit expressions for $K_H(t, s)$, we refer to    \cite[Chapter 2.1]{BHOZ} and  \cite[Section 3]{DU99}.}
    
 {
  By Theorem 3.3 of \cite{DU99},   the Cameron-Martin space $\mathcal H$    can be represented  as }
\begin{equation}\label{eq CM}
 { \mathcal H=\left\{\xi: [0,1]\rightarrow \mathbb R;\, \xi(t)=\int_0^t K_H(t, u) \dot \xi(u)du  \ \text{for  } \dot \xi\in L^2([0,1])\right\}. }
\end{equation}
 {
The inner product on $\mathcal H$ is given by }
\begin{equation}\label{eq inn 2}  {
\langle \xi_1, \xi_2\rangle_{\mathcal H}=\int_0^1 \dot \xi_1(u) \dot \xi_2(u) du.}
\end{equation}
 {
Let $\mathscr S$ be the unit ball in $\mathcal H$, that is 
$$
\mathscr S=\left\{\xi\in \mathcal H;\, \|\xi\|_{\mathcal H}\le 1 \right\}.
$$}

 {
The following result is needed to determine the constants in Example \ref{exa LIL}.  }
\begin{lemma}\label{lem const}
\begin{itemize}
\item[(a)]  { It holds that }
\begin{align*}  {
\sup_{\xi \in \mathscr S}\xi(1) =\sup_{\xi \in \mathscr S}\sup_{0\le t\le 1}|\xi(t)|=1.  }
\end{align*}
\item[(b)]   {For $0<\delta<1$, we have }
\begin{align*}  {
\sup_{\xi \in \mathscr S}\sup_{0\le t\le 1-\delta}|\xi(t+\delta)-\xi(t)|   = \sup_{\xi \in \mathscr S}\sup_{0\le t\le 1-\delta}\sup_{0\le s\le \delta}|\xi(t+s)-\xi(t)| \in \left[\delta^H, \sqrt{2}\delta^H\right]. }
\end{align*}
 \end{itemize}
\end{lemma}
 \begin{proof}
\noindent  {  (a)  For any $\xi \in \mathscr S$, by the Cauchy-Schwarz inequality and \eqref{eq cov 2}, we have 
\begin{align*}
\xi(1)=\int_0^1 K_H(1, u) \dot \xi(u)du\le &\,   \sqrt{ \int_0^1 K_H(1, u)^2  du \cdot \int_0^1  |\dot \xi(u)|^2du }\\
 \le & \,  \sqrt{ \int_0^1 K_H(1, u)^2  du} = 1. 
\end{align*}
Moreover, the above equalities holds if we choose $\dot \xi(u)=K_H(1, u)$ on   $[0,1]$.  Thus,  $$\sup_{\xi \in \mathscr S}\xi(1)=1.$$
}

 {
Similarly,   for any $\xi \in \mathscr S$, by the Cauchy-Schwarz inequality and \eqref{eq cov 2}, we have  that for any $t\in [0, 1]$,
\begin{align*}
|\xi(t)|=\left|\int_0^t K_H(t, u) \dot \xi(u)du\right|
\le &\,   \sqrt{ \int_0^t K_H(t, u)^2  du \cdot \int_0^t  |\dot \xi(u)|^2du }\\
\le &\,  \sqrt{ \int_0^t K_H(t, u)^2  du}\le 1. 
\end{align*}
If we choose $\dot \xi(u)=K_H(1, u)$ on $[0,1]$, then we have  $$\sup_{\xi \in \mathscr S}\sup_{0\le t\le 1}|\xi(t)|=1.$$
}
\noindent
  (b)  {Given $0<\delta<1$ and $\xi\in\mathscr S$,  by the Cauchy-Schwarz inequality and \eqref{eq cov 2}, we have that for any $t\in [0, 1-\delta]$, 
 \begin{align*}
 |\xi(t+\delta)-\xi(t)|= &\, \left|\int_0^{t+\delta} K_H(t+\delta, u) \dot \xi(u)du-\int_0^{t } K_H(t, u) \dot \xi(u)du \right|\\
 = &\, \left|\int_0^{t}\left( K_H(t+\delta, u)-K_H(t, u)\right)  \dot \xi(u)du \right| +\left|\int_t^{t+\delta} K_H(t+\delta, u) \dot \xi(u)du  \right|\\
 \le &\,  \left(\int_0^{t}\left( K_H(t+\delta, u)-K_H(t, u)\right)^2 du\right)^{\frac12}+  \left(\int_{t}^{t+\delta} K_H(t+\delta, u) ^2 du\right)^{\frac12}.
\end{align*}
}
 { Moreover,  the equality in the last step is achieved when
 \begin{equation*} 
  \dot\xi(u)= \left\{\begin{array}{ll}
K_H(t+\delta, u)-K_H(t, u),  \ \quad &\hbox{ when }\  u\in [0, t],\\
 K_H(t+\delta, u),  \ \quad &\hbox{ when } \ u\in (t, t+\delta].
\end{array}
\right.
 \end{equation*}
 }
 {By \eqref{eq cov 1} and \eqref{eq cov 2}, we know 
 $$
 \int_0^{t}\left( K_H(t+\delta, u)-K_H(t, u)\right)^2 du +  \int_{t}^{t+\delta} K_H(t+\delta, u) ^2 du=\delta^{2H}.
 $$
 Consequently, using the elementary inequality
  $$ \sqrt{a+b}\le \sqrt{a}+\sqrt{b}\le \sqrt{2(a+b)}, \ \  a, b\ge0, $$
 we have
 $$
\left(\int_0^{t}\left( K_H(t+\delta, u)-K_H(t, u)\right)^2 du\right)^{\frac12}+  \left(\int_{t}^{t+\delta} K_H(t+\delta, u) ^2 du\right)^{\frac12} \in \left[\delta^H, \sqrt{2}\delta^H\right].
 $$
 A parallel argument yields
  $$
   \sup_{\xi \in \mathscr S}\sup_{0\le t\le 1-\delta}\sup_{0\le s\le \delta}|\xi(t+s)-\xi(t)| \in  \left[\delta^H, \sqrt{2}\delta^H\right].
  $$
  In particular,   when $H=\frac12$,   substituting \eqref{eq KH1} into the above estimates leads to the following  exact value:
  \begin{align*}
\sup_{\xi \in \mathscr S}\sup_{0\le t\le 1-\delta}|\xi(t+\delta)-\xi(t)|   = \sup_{\xi \in \mathscr S}\sup_{0\le t\le 1-\delta}\sup_{0\le s\le \delta}|\xi(t+s)-\xi(t)| = \delta^{\frac12}.
\end{align*}
 }
  {The proof is complete. }
 \end{proof}
 
\vskip0.5cm

\noindent{\bf Acknowledgments}  \ \   {The authors sincerely thank the reviewers for their insightful comments and careful review, 
which have helped to improve the paper significantly.} The research of R. Wang is partially supported by the NSF of Hubei Province (2024AFB683). 
The research of Y. Xiao is partially supported by the NSF grant DMS-2153846.

\end{document}